\documentclass{amsart}

\usepackage{amsmath}
\usepackage{amsthm}
\usepackage{amssymb}
\usepackage{enumitem}
    \setlist[enumerate]{label=\textnormal{(\arabic*)}} 
\usepackage[hidelinks]{hyperref}
\usepackage[capitalise]{cleveref}
\usepackage{mathtools}
\usepackage{nameref}
\usepackage[new]{old-arrows}
\usepackage{stmaryrd}
\usepackage{thmtools}
\usepackage[minimal]{yhmath}

\usepackage{tikz}
\usetikzlibrary{cd}
\usetikzlibrary{positioning, angles, quotes}

\theoremstyle{plain}
\newtheorem{thm}{Theorem}[section]      \newtheorem*{thm*}{Theorem}
\newtheorem{cor}[thm]{Corollary}        \newtheorem*{cor*}{Corollary}
\newtheorem{prop}[thm]{Proposition}     \newtheorem*{prop*}{Proposition}
\newtheorem{lem}[thm]{Lemma}            \newtheorem*{lem*}{Lemma}
          \newtheorem*{claim*}{Claim}
        \newtheorem*{exer*}{Exercise}
\newtheorem{q}[thm]{Question}           \newtheorem*{q*}{Question}
      \newtheorem*{conj*}{Conjecture}

\theoremstyle{definition}
\newtheorem{defn}[thm]{Definition}      \newtheorem*{defn*}{Definition}
           \newtheorem*{ex*}{Example}

\theoremstyle{remark}
\newtheorem{rem}[thm]{Remark}           \newtheorem*{rem*}{Remark}
     \newtheorem*{conv*}{Conventions}

\newtheoremstyle{iremark}
    {0.5em\topsep}   
    {0.5em\topsep}   
    {\upshape}  
    {0pt}       
    {\itshape}  
    {.}         
    {5pt plus 1pt minus 1pt} 
    {\thmname{#1}\thmnumber{ \itshape#2}\thmnote{ (#3)}} 
\theoremstyle{iremark}

\theoremstyle{plain}
\newtheorem{manualtheoreminner}{Theorem}
\newenvironment{manualtheorem}[1]{
    \renewcommand\themanualtheoreminner{#1}
    \manualtheoreminner
}{\endmanualtheoreminner}

\newtheorem{manualcorinner}{Corollary}\newenvironment{manualcor}[1]{\renewcommand\themanualcorinner{#1}\manualcorinner}{\endmanualcorinner}

\Crefname{thm}{Theorem}{Theorems}
\Crefname{defn}{Definition}{Definitions}
\Crefname{claim}{Claim}{Claims}
\Crefname{ex}{Example}{Examples}
\Crefname{prop}{Proposition}{Propositions}
\Crefname{manualtheoreminner}{Theorem}{Theorems}

\newcommand{\BS}{\mathrm{BS}}

\newcommand{\PU}{\mathrm{PU}}

\newcommand{\FP}{\mathrm{FP}}

\renewcommand{\b}[1]{b^{(2)}_{#1}}

\newcommand{\Dk}[1]{\mathcal D_{k[{#1}]}}

\newcommand{\nov}[1]{\widehat{k[{#1}]}^\chi}

\newcommand{\N}{\mathbb{N}}
\newcommand{\Q}{\mathbb{Q}}
\newcommand{\R}{\mathbb{R}}
\newcommand{\Z}{\mathbb{Z}}

\newcommand{\inv}{^{-1}}

\DeclareMathOperator{\cd}{cd}

\DeclareMathOperator{\Ext}{Ext}

\let\H\relax
\DeclareMathOperator{\H}{H}
\DeclareMathOperator{\hd}{hd}
\DeclareMathOperator{\Hom}{Hom}

\DeclareMathOperator{\Ore}{Ore}

\DeclareMathOperator{\Tor}{Tor}

\makeatletter
\newsavebox{\@brx}
\newcommand{\llangle}[1][]{\savebox{\@brx}{\(\m@th{#1\langle}\)}%
  \mathopen{\copy\@brx\mkern2mu\kern-0.9\wd\@brx\usebox{\@brx}}}
\newcommand{\rrangle}[1][]{\savebox{\@brx}{\(\m@th{#1\rangle}\)}%
  \mathclose{\copy\@brx\mkern2mu\kern-0.9\wd\@brx\usebox{\@brx}}}
\makeatother

\newcounter{comments}

\title{On the cohomological dimension of kernels of maps to $\Z$}
\author{Sam P.~Fisher}
\address{University of Oxford, Oxford, OX2 6GG, UK}
\email{sam.fisher@maths.ox.ac.uk}

\begin{document}

\begin{abstract}
    For a group \(G\) of type \(\FP(R)\) for a ring \(R\) and a homomorphism \(\chi \colon G \rightarrow \Z\), we show that \(\cd_R(\ker \chi) = \cd_R(G) - 1\) if the top-dimensional cohomology of \(G\) with coefficients in the Novikov rings \(\widehat{R[G]}^{\pm\chi}\) vanishes. This criterion is applied to show that if $G$ is a finitely generated RFRS group of cohomological dimension $2$, then $G$ is virtually free-by-cyclic if and only if $\b{2}(G) = 0$. This answers a question of Wise and generalises and gives a significantly shorter proof of a recent theorem of Kielak and Linton, where the same result is obtained under the additional hypotheses that $G$ is virtually compact special and hyperbolic. A consequence of the result is that all virtually RFRS groups of rational cohomological dimension \(2\) with vanishing second \(\ell^2\)-Betti number are coherent. More generally, we show that if $G$ is a RFRS group of cohomological dimension $n$ and of type $\FP_{n-1}(\Q)$, then $G$ admits a virtual map to $\Z$ with kernel of rational cohomological dimension $n-1$ if and only if $\b{n}(G) = 0$.
\end{abstract}

\maketitle

\section{Introduction}

There is an emerging connection between the coherence of two-dimensional groups and the vanishing of the second $\ell^2$-Betti number. Indeed, there are now many results showing that if $G$ is a group of cohomological dimension $2$ and $\b{2}(G) = 0$, then $G$ is coherent \cite{Wise_secondL2vanish, KielakKrophollerWilkes_RandomL2, KielakLinton_FbyZ, JaikinLinton_coherence}. Moreover, there are no known examples of coherent groups with non-vanishing second $\ell^2$-Betti number. Related to coherence is Wise's notion of nonpositive immersions: a $2$-complex $X$ has \emph{nonpositive immersions} if for every immersion $Y \looparrowright X$ of a compact connected complex $Y$, either $\chi(Y) \leqslant 0$ or $\pi_1(Y)$ is trivial. Wise conjectures that if $X$ is aspherical and has nonpositive immersions, then $\pi_1(X)$ is coherent \cite[Conjecture 12.11]{Wise_anInvitation}, and attributes to Gromov the observation that the nonpositive immersions property should be connected to the vanishing of $\b{2}(\widetilde{X})$ (see \cite[Section 16]{Wise_anInvitation} and \cite{WiseEnergy}). Wise even conjectures that having nonpositive immersions should be equivalent to the vanishing of the second $\ell^2$-Betti number \cite[Conjecture 2.6]{WiseEnergy}. It thus makes sense to ask the following question.

\begin{q}\label{q:intro_question}
    Let $G$ be a finitely generated group of geometric dimension at most two. Are the following properties equivalent?
    \begin{enumerate}
        \item $\b{2}(G) = 0$.
        \item $G$ is coherent.
        \item $G$ has a finite-index subgroup that is the fundamental group of a $2$-complex with nonpositive immersions.
    \end{enumerate}
\end{q}

In this article, we focus on the class of residually finite rationally solvable (RFRS) groups, which were defined by Agol in connection with Thurston's virtual fibring conjecture \cite{AgolCritVirtFib}. Notably, the class of RFRS groups contains all special groups, introduced by Haglund and Wise \cite{HaglundWise_special}, which are fundamental groups of particularly well-behaved nonpositively curved cube complexes. Special groups provide a rich source of examples of RFRS groups, though there are also interesting RFRS groups which are not virtually special ($\Z \wr \Z$ is an example--see also \cite{AgolStover_RFRS}, which gives an example of a RFRS lattice in $\PU(2,1)$).

A group $G$ is \emph{free-by-cyclic} if it fits into a short exact sequence 
\[
    1 \to F \to G \to Z \to 1,
\]
where $F$ is free (we emphasise that $F$ need not be finitely generated) and $Z$ is cyclic. Our main result characterises which RFRS groups are virtually free-by-cyclic.

\begin{manualtheorem}{A}\label{thm:main}
    Let $G$ be a finitely generated RFRS group. Then $G$ is virtually free-by-cyclic if and only if $\b{2}(G) = 0$ and $\cd_\Q(G) \leqslant 2$.
\end{manualtheorem}

The same result was recently obtained by Kielak and Linton, under the additional hypotheses that \(G\) be hyperbolic and virtually compact special \cite[Theorem 1.1]{KielakLinton_FbyZ}, and they speculate as to whether the hyperbolicity assumption can be dropped, and even whether the result holds for RFRS groups in general (see \cite[p.~1581]{KielakLinton_FbyZ}). We confirm that this is the case here, and our proof has the advantage of being significantly shorter. Moreover, the more algebraic methods we use can be applied to obtain analogous results for strictly larger classes of groups; in a forthcoming article of the author and Kevin Klinge, we show that a finitely generated residually (poly-\(\Z\) virtually nilpotent) group \(G\) with \(\cd_\Q(G) \leqslant 2\) is free-by-(poly-\(\Z\) virtually nilpotent) if and only if \(\b{2}(G) = 0\), and we obtain the same applications to coherence as those mentioned below. This applies, in particular, to residually torsion-free nilpotent groups, a class of groups which also contains all special groups \cite{DuchampKrob_RAAGsRTFN}. Note that finitely generated RFRS groups are residually (poly-\(\Z\) virtually Abelian) \cite[Theorem 6.3]{OkunSchreve_DawidSimplified}. 

Being virtually free-by-cyclic is a strong form of coherence. Indeed, if $G$ is free-by-cyclic, then $G$ is coherent by a result of Feighn and Handel \cite{FeighnHandel_FreeByZCoherent} (and coherence is a commensurability invariant). Wise proved that $G$ is the fundamental group of a $2$-complex with nonpositive immersions \cite[Theorem 6.1]{Wise_JussieuCoherenceNPI}, and Henneke--López-Álvarez showed that $k[G]$ is a pseudo-Sylvester domain for any field $k$, and is therefore coherent \cite[Theorem B, Proposition 2.9]{HennekeLopez_PseudoSyl} (meaning that all of its finitely generated one-sided ideals are finitely presented). We thus obtain the following corollary of \cref{thm:main}, giving the implications (1) $\Rightarrow$ (2) and (1) $\Rightarrow$ (3) of \cref{q:intro_question} in the class of RFRS groups.

\begin{manualcor}{B}
    Let $G$ be a RFRS group with $\cd_\Q(G) \leqslant 2$. If $\b{2}(G) = 0$, then
    \begin{enumerate}
        \item $G$ is coherent;
        \item the group algebra $k[G]$ is a coherent ring for any field $k$;
        \item if $G$ is finitely generated, then $G$ is virtually the fundamental group of an aspherical $2$-complex with nonpositive immersions.
    \end{enumerate}
\end{manualcor}

Determining when two-dimensional groups are virtually free-by-cyclic is an interesting problem in its own right, and there are many results in the literature addressing this question. However, the methods used in each case are somewhat ad hoc and often implicitly rely on the vanishing of the first $\ell^2$-Betti number at some point in the arguments. We list some applications of \cref{thm:main}, which provides a uniform treatment of many of these results.

\begin{manualcor}{C}\label{cor:apps}
    Let $G$ be a RFRS group and additionally suppose that one of the following holds:
    \begin{enumerate}
        \item\label{item:EH} $G$ admits an elementary hierarchy \cite[Theorem A]{HagenWise_freebyZ};
        \item\label{item:2cx} $G$ is the fundamental group of a finite $2$-complex $X$ such that $b_2(X) = 0$;
        \item\label{item:1rel} $G$ is a one-relator group.
    \end{enumerate}
    Then $G$ is virtually free-by-cyclic.
\end{manualcor}

\cref{item:EH} was first obtained by Hagen and Wise in \cite{HagenWise_freebyZ}; a group has an elementary hierarchy of length $0$ if it is trivial, and an elementary hierarchy of length $n$ if it splits as a graph of groups with vertex groups admitting an elementary hierarchy of length $n-1$ and with cyclic edge groups. It is easy to show, by induction on the length of the hierarchy, that groups with an elementary hierarchy are two-dimensional and have vanishing second $\ell^2$-Betti number. Two-dimensional limit groups (equivalently, limit groups without $\Z^3$-subgroups) and graphs of free groups with cyclic edge groups that do not contain any Baumslag--Solitar subgroups $\BS(m,n)$ with $|m| \neq |n|$ are examples of virtually RFRS groups with elementary hierarchies (\cite[Corollary 18.3]{Wise_structureQCH} and \cite{HsuWise_graphOfFree}). In the case of limit groups we will be able to conclude something more general: if $G$ is a limit group with $\cd_\Q(G) = n$, then $G$ admits a virtual map to $\Z$ with kernel of rational cohomological dimension $n-1$ (this follows from \cite[Corollary C]{BridKoch_volumeGradient} and \cref{thm:main} below).

In \cite[Theorem 1.1, Corollary 6.2]{Wise_secondL2vanish}, Wise proves that if $X$ is a compact $2$-complex with $\pi_1(X)$ RFRS, then $\b{1}(\widetilde X) \leqslant b_1(X) - 1$ and $\b{2}(\widetilde X) \leqslant b_2(X)$. He then uses Kielak's theorem \cite[Theorem 5.4]{KielakRFRS} together with a result of Feldman \cite[Theorem 2.4]{Feldman71} to conclude that if $b_1(X) = 1$ and $b_2(X) = 0$, then $\pi_1(X)$ is virtually free-by-cyclic. In \cite[Problem 6.5]{Wise_secondL2vanish}, Wise asks whether the assumption $b_1(X) = 1$ is necessary, and item \ref{item:2cx} of \cref{cor:apps} confirms that it is not. The assumptions that $\pi_1(X)$ is RFRS and that $b_2(X) = 0$ together imply that $X$ is aspherical (see the proof of \cite[Proposition 6.4]{Wise_secondL2vanish}), which is why we do not include an asphericity assumption in the statement of item \ref{item:2cx}.

Since one-relator groups are of rational cohomological dimension at most $2$ and have vanishing second $\ell^2$-Betti number \cite{DicksLinnell2007}, we see that item \ref{item:1rel} follows immediately from \cref{cor:apps}. The question of which one-relator groups are virtually free-by-cyclic has received much attention. Baumslag conjectured that one-relator groups with torsion are virtually free-by-cyclic \cite[Problem 6]{Baumslag_oneRelSurvey}, and this was later strengthened by Wise who made the same conjecture for all hyperbolic one-relator groups. In \cite[Theorem 17.11]{Wise_anInvitation}, Wise observes that every virtually special two-generator one-relator group is virtually free-by-cyclic, again as a consequence of $\ell^2$-acyclicity and the results of \cite{KielakRFRS} and \cite{Feldman71}. Recently, Kielak and Linton showed that if $G$ is hyperbolic and virtually compact special, then $G$ is virtually free-by-cyclic if and only if $\b{2}(G) = 0$ and $\cd_\Q(G) \leqslant 2$ \cite[Theorem 1.1]{KielakLinton_FbyZ}. Since one-relator groups with torsion are hyperbolic \cite{Newman_spelling} and virtually compact special \cite[Corollary 19.2]{Wise_structureQCH}, Kielak and Linton's theorem resolves Baumslag's conjecture. Note that Wise's conjecture remains open as hyperbolic one-relator groups are not known to be virtually special.

\subsection{Summary of the proof}

\cref{thm:main} and \cite[Theorem 1.1]{KielakLinton_FbyZ} are both special cases of results that take as input a group of rational cohomological dimension $n$ and produce a virtual map to $\Z$ with kernel of rational cohomological dimension $n-1$ (\cref{thm:weak_fibre} below and Theorem 1.11 in \cite{KielakLinton_FbyZ}). The results about free-by-cyclic groups then follow by applying these theorems at $n = 2$ and appealing to the Stallings--Swan theorem \cite{Stallings_cd1,Swan_cd1}. However, the methods in each case are quite different; we review them briefly here.

Suppose that $G$ is a compact special hyperbolic group of cohomological dimension $n > 1$, and assume that $\b{1}(G) > 0$ and that $\b{i}(G) = 0$ for all $i > 1$. In \cite{KielakLinton_FbyZ}, it is shown that $G$ embeds in an HNN extension $H = G*_F$ such that $\cd_\Q(H) = \cd_\Q(G)$ and $H$ is $\ell^2$-acyclic. Moreover, the hyperbolicity and specialness assumptions are used to show that $H$ can be arranged to be virtually compact special, and therefore admits a virtual map to $\Z$ with kernel $N$ of type $\FP(\Q)$ by \cite[Theorem A]{Fisher_Improved}. Using Feldman's theorem \cite[Theorem 2.4]{Feldman71}, we conclude that $\cd_\Q(N) = n-1$. Restricting the virtual map to $G$, we conclude that $G$ admits a virtual map to $\Z$ with kernel of rational cohomological dimension $n - 1$.
 
As the proof given here is entirely homological, we do not need the geometric assumptions of hyperbolicity and specialness, but only the more algebraic RFRS condition. Moreover, we will only need to assume that the top-dimensional $\ell^2$-Betti number vanishes (as opposed to all $\ell^2$-Betti numbers in dimensions greater than $1$) since the proof does not rely on algebraic fibring, contrasting with the approach taken by Kielak and Linton. Thus, we take as input a RFRS group $G$ of finite type (in fact, type $\FP_{n-1}(\Q)$ suffices) and of cohomological dimension $n$ and assume that $\b{n}(G) = 0$. In \cite{KielakRFRS}, Kielak shows that an $\ell^2$-acyclic RFRS group of finite type has a finite-index subgroup $H$ whose homology with coefficients in the \emph{Novikov ring} $\widehat{\Q[H]}^\chi$ vanishes for many maps $\chi \colon H \rightarrow \Z$ (the Novikov ring is a certain completion of the group algebra $\Q[H]$ with respect to $\chi$; see \cref{def:nov}). We observe that the assumption $\b{n}(G) = 0$ also implies that the Novikov \emph{co}homology $\H^n(H; \widehat{\Q[H]}^\chi)$ vanishes for many maps $\chi$ and that we do not need \(\ell^2\)-acyclicity in all dimensions. The main theorem then follows from the following new general criterion involving Novikov cohomology.

\begin{manualtheorem}{D}\label{thm:novikov_criterion}
    Let \(R\) be a ring and let $G$ be a group of type $\FP(R)$ with $\cd_R(G) = n$. Suppose that $\chi \colon G \rightarrow \R$ is a nontrivial character such that
    \[
        \H^n(G; \widehat{R[G]}^{\pm\chi}) = 0.
    \]
    Then $\cd_R(\ker \chi) < n$.
\end{manualtheorem}

It is interesting to compare this criterion with Sikorav's theorem \cite{SikoravThese} (see also \cite[Theorem 3.11]{KielakRFRS}), which states that $\H_1(H; \widehat{\Q[H]}^{\pm\chi}) = 0$ if and only if $\ker \chi$ is finitely generated. Our main result generalises to higher degrees and to fields of positive characteristic as follows.

\begin{manualtheorem}{E}\label{thm:weakFibre}
    Let $k$ be a field and let $G$ be a nontrivial RFRS group of type $\FP_{n-1}(k)$ with $\cd_k(G) = n$. The following are equivalent:
    \begin{enumerate}
        \item $\b{n}(G;k) = 0$;
        \item there is a finite-index subgroup $H \leqslant G$ and an epimorphism $\chi \colon H \rightarrow \Z$ such that $\cd_k(\ker \chi) = n-1$;
        \item there is a finite-index subgroup $H \leqslant G$ and an epimorphism $\chi \colon H \rightarrow \Z$ such that $\hd_k(\ker \chi) = n-1$.
    \end{enumerate}
\end{manualtheorem}

We refer the reader to \cref{thm:cd_drop_nov,thm:weak_fibre} for sharpened versions of \cref{thm:novikov_criterion,thm:weakFibre}, where the assumed finiteness conditions are substantially weakened. For now, we mention that $\b{i}(G; \Q)$ is the usual $\ell^2$-Betti number $\b{i}(G)$. Note that if the $\ell^2$-Betti numbers vanish in low dimensions as well as in the top dimension, then we can deduce finiteness properties of the kernel (see \cref{thm:weakFibre} for a precise statement). This may be of interest because the $\ell^2$-Betti numbers of many groups of classical interest are concentrated in their middle dimension. 

In \cref{sec:mains}, we will use \cref{thm:weakFibre} to deduce higher coherence properties of group algebras of RFRS groups with vanishing top-dimensional $\ell^2$-Betti numbers (see \cref{cor:group_alg_coherence}).

\subsection{Acknowledgments}

The author is grateful to Dawid Kielak for many useful conversations and comments on this article and to Andrei Jaikin-Zapirain for communicating simplifications of the original argument that have led to a far more transparent and streamlined proof. The author thanks the referee for their useful comments, which have improved the exposition. The author is supported by the National Science and Engineering Research Council (NSERC) (ref.~no.~567804-2022) and the European Research Council (ERC) under the European Union's Horizon 2020 research and innovation programme (Grant agreement No. 850930).

\section{Preliminaries} \label{sec:prelims}

Throughout the article, $k$ always denotes a field and rings are always assumed to be unital, associative, and nonzero.

\subsection{Division rings}

A group is \emph{locally indicable} if all its nontrivial finitely generated subgroups admit an epimorphism to $\Z$. Consider the group algebra $k[G]$ of a locally indicable group $G$. Let $\mathcal D$ be a division ring such that there is an embedding $\varphi \colon k[G] \hookrightarrow \mathcal D$. This makes $\mathcal D$ into a $k[G]$-bimodule. If $N$ is a subgroup of $G$, we denote the division closure of $\varphi(k[N])$ in $\mathcal D$ by $\mathcal D_N$. The embedding $\varphi$ is \emph{Hughes-free} if, for each finitely generated subgroup $H \leqslant G$ and each normal subgroup $N \trianglelefteqslant H$ such that $H/N \cong \Z$, the multiplication map
\[
    \mathcal D_N \otimes_{k[N]} k[H] \rightarrow \mathcal D,
\]
defined by $x \otimes y \mapsto x \cdot \varphi(y)$ on elementary tensors, is injective. Hughes proved that if $G$ is locally indicable and $k[G]$ admits a Hughes-free embedding $\varphi \colon k[G] \hookrightarrow \mathcal D$, then $\mathcal D$ is unique up to $k[G]$-isomorphism \cite{HughesDivRings1970}. Thus, we denote the \emph{Hughes-free division ring} of $k[G]$ by $\Dk{G}$. If $H \leqslant G$ is any subgroup, then $\mathcal D_H \cong \mathcal D_{k[H]}$. We will usually omit the map $\varphi$ from the notation and think of $k[G]$ as a subring of $\Dk{G}$.

We will consider the homology and cohomology of a Hughes-free embeddable group $G$ with coefficients in $\Dk{G}$. We define
\[
    \H_n^{(2)}(G; k) = \Tor_n^{k[G]}(k,\Dk{G}) \quad \text{and} \quad \H^n_{(2)}(G; k) = \Ext_{k[G]}^n(k, \Dk{G}).
\]
Since modules over division rings have well-defined dimensions, we can define the Betti numbers
\[
    b_n^{(2)}(G;k) = \dim_{\Dk{G}} \H_n^{(2)}(G; k) \quad \text{and} \quad b^n_{(2)}(G;k) = \dim_{\Dk{G}} \H^n_{(2)} (G;k).
\]
It is easy to see that if both $b_n^{(2)}(G;k)$ and $b^n_{(2)}(G;k)$ are finite, then they are equal (see \cite[Lemma 2.2]{KielakLinton_FbyZ}).

\subsection{RFRS}

Residually finite rationally solvable (RFRS) groups were introduced by Agol in \cite{AgolCritVirtFib}, where he showed that a compact irreducible $3$-manifold $M$ with $\chi(M) = 0$ virtually fibres over the circle provided that $\pi_1(M)$ is virtually RFRS. A group $G$ is \emph{RFRS} if there is a chain $G = G_0 \geqslant G_1 \geqslant \dots$ of finite-index subgroups such that $\bigcap_{i \geqslant 0} G_i = \{1\}$ and $\ker(G_i \rightarrow G_i^{\mathsf{ab}} \otimes \Q) \leqslant G_{i+1}$ for all $i \geqslant 0$. The main source of examples of RFRS groups are special groups, defined by Haglund and Wise \cite{HaglundWise_special}. A nonpositively curved cube complex is special if it avoids certain pathological hyperplane configurations, and a group is (compact) special if it is the fundamental group of a (compact) special cube complex. Equivalently, a finitely generated group is special if and only if it is isomorphic to a subgroup of a right-angled Artin group. We refer the reader to the original paper of Haglund and Wise for more details.

RFRS groups are locally indicable and therefore satisfy the strong Atiyah conjecture (see \cite{JaikinLopezStrongAtiyah2020}, though this can also be deduced from earlier work of Schick \cite{SchickL2Int2002}). The consequence of the strong Atiyah conjecture that interests us is that $\Q[G]$ has a Hughes-free embedding. More precisely, the division closure $\mathcal D(G)$ of $\Q[G]$ in the algebra of affiliated operators $\mathcal U(G)$ is a Hughes-free division ring, and moreover the $\ell^2$-Betti numbers of $G$ are can be computed via
\[
    \b{n}(G) = \dim_{\mathcal D(G)} \Tor_n^{\Q[G]}(\Q,\mathcal D(G))
\]
by \cite[Proof of Lemma 3.7]{LinnellDivRings93} and \cite[Lemma 10.28]{Luck02}. By the uniqueness of Hughes-free division rings, $\mathcal D(G) \cong \mathcal D_{\Q[G]}$, and therefore $\b{n}(G;\Q)$ equals the usual $\ell^2$-Betti number $\b{n}(G)$. In \cite[Corollary 1.3 and Proposition 4.4]{JaikinZapirain2020THEUO}, it is shown that if $G$ is RFRS, then the Hughes-free division ring $\Dk{G}$ exists for all fields $k$. Thus, we will consider all the Betti numbers $\b{n}(G;k)$ below.

\subsection{Finiteness properties}

Let $R$ be a ring. An $R$-module $M$ is of \emph{type $\FP_n$} if there is a projective resolution $\cdots \rightarrow P_1 \rightarrow P_0 \rightarrow M \rightarrow 0$ such that $P_i$ is finitely generated for all $i \leqslant n$. If $P_i$ is finitely generated for all $i$ we say that $M$ is of \emph{type $\FP_\infty$}. If, additionally, there exists $N$ such that $P_i = 0$ for all $i > N$ then $M$ is of \emph{type $\FP$}. A group $G$ is of \emph{type $\FP_n(R)$} (resp.~$\FP_\infty(R)$, $\FP(R)$) if the trivial $R[G]$-module $R$ is of type $\FP_n$ (resp.~$\FP_\infty$, $\FP$). Note that $G$ is finitely generated if and only if $G$ is of type $\FP_1(R)$ for some (and hence every) ring $R$.

A group $G$ is of \emph{cohomological (resp.~homological) dimension $n$ over $R$} if the trivial $R[G]$-module $R$ admits a projective (resp.~flat) resolution of length $n$, and it does not admit a projective (resp.~flat) resolution of length $n-1$. In this case, we write $\cd_R(G) = n$ (resp.~$\hd_R(G) = n$). If no such $n$ exists, then we define $\cd_R(G) = \infty$ (resp.~$\hd_R(G) = \infty$). The following result shows that the (co)homological dimension of a group is detected by group (co)homology with coefficients in an appropriate module.

\begin{prop}[{\cite[Proposition 4.1 a) and b)]{Bieri_QueenMary}}]\label{prop:hd_cd}
    Let $G$ be a group and let $R$ be a ring.
    \begin{enumerate}
        \item $\hd_R(G) < n$ if and only if $\H_i(G;M) = 0$ for all $R[G]$-modules $M$ and all $i \geqslant n$.
        \item $\cd_R(G) < n$ if and only if $\H^i(G;M) = 0$ for all $R[G]$-modules $M$ and all $i \geqslant n$.
    \end{enumerate}
\end{prop}

\subsection{Higher coherence}

If $R$ is a ring, then a group is \emph{homologically $n$-coherent over $R$} if every subgroup of type $\FP_n(R)$ is of type $\FP_{n+1}(R)$. When $n = 1$ and $R = \Z$, this property is called \emph{homological coherence}. The ring $R$ is \emph{(left) $n$-coherent} if every (left) ideal of type $\FP_n$ is of type $\FP_{n+1}$. It is not hard to see that if $k[G]$ is $n$-coherent, then $G$ is homologically $n$-coherent over $k$. It is open whether the reverse implication holds, even for $n = 1$.

A ring $R$ is of \emph{left global dimension at most $n$} if every left $R$-module $M$ has a projective resolution of length at most $n$. If $G$ is a group with $\cd_k(G) \leqslant n$, then $k[G]$ is of global dimension at most $n$ \cite[Proposition 2.2]{JaikinLinton_coherence}. The \emph{weak dimension} of a right $R$-module $M$ is the supremal $n$ such that there exists a left $R$-module $N$ and $\Tor_n^R(M,N) \neq 0$. All these concepts can be defined with the words ``left" and ``right" interchanged.

The following results are proven in \cite[Section 3]{JaikinLinton_coherence} in the case $n = 2$; we include proofs for the convenience of the reader. Both results are based on the following key observation \cite[Lemma 3.1]{JaikinLinton_coherence}: if $P$ is a projective $R$-module such that $R$ embeds into a division ring $\mathcal D$, then $P$ is finitely generated if and only if $P \otimes_R \mathcal D$ is finitely generated as a $\mathcal D$-module.

\begin{prop}[{\cite[Corollary 3.3]{JaikinLinton_coherence}}]\label{prop:ring_coherence}
    Let $R$ be a ring of right global dimension at most $n$ that embeds into a division ring $\mathcal D$ of weak dimension at most $n-1$ as a left $R$-module. Then $R$ is right $(n-1)$-coherent.
\end{prop}
\begin{proof}
    Let  $I \trianglelefteqslant R$ be a right ideal of type $\FP_{n-1}$, and fix a partial resolution
    \[
        P_{n-1} \rightarrow \cdots \rightarrow P_1 \rightarrow P_0 \rightarrow I \rightarrow 0
    \]
    by finitely generated projective right $R$-modules. The inclusion $I \hookrightarrow R$ and the quotient $R \rightarrow R/I$ induce a partial projective resolution 
    \[
        P_{n-1} \rightarrow \cdots \rightarrow P_1 \rightarrow P_0 \rightarrow R \rightarrow R/I \rightarrow 0
    \]
    of $R/I$. But $R$ is of global dimension at most $n$, which implies that $R/I$ is of projective dimension at most $n$. Therefore, the kernel $P_n$ of $P_{n-1} \rightarrow P_{n-2}$ is projective by Schanuel's lemma, so
    \[
        0 \rightarrow P_n \rightarrow P_{n-1} \rightarrow \cdots \rightarrow P_0 \rightarrow I \rightarrow 0
    \]
    is a projective resolution of $I$. By definition, the sequence
    \[
        0 \rightarrow \Tor_n^R(I, \mathcal D) \rightarrow P_n \otimes_R \mathcal D \rightarrow P_{n-1} \otimes_R \mathcal D
    \]
    is exact. But $\Tor_n^R(I, \mathcal D) = 0$, since $\mathcal D$ is of weak dimension at most $n-1$. Therefore, $P_n \otimes_R \mathcal D$ injects into $P_{n-1} \otimes_R \mathcal D$, and since $P_{n-1}$ is finitely generated and $\mathcal D$ is a division ring, this implies that $P_n \otimes_R \mathcal D$ is a finite-dimensional vector space over $\mathcal D$. By \cite[Lemma 3.1]{JaikinLinton_coherence}, this implies that $P_n$ is finitely generated, and therefore that $I$ is of type $\FP_n$. Since $I$ was arbitrary, this shows that $R$ is right coherent.
\end{proof}

\begin{prop}[{\cite[Theorem 3.10]{JaikinLinton_coherence}}]\label{prop:hom_n_coherence}
    Let $G$ be a locally indicable group with $\cd_k(G) = n$ such that the Hughes-free division ring $\Dk{G}$ exists. If $\b{n}(G;k) = 0$, then $G$ is homologically $(n-1)$-coherent over $k$.
\end{prop}
\begin{proof}
    Let $H \leqslant G$ be a subgroup of type $\FP_{n-1}(k)$. Since $\cd_k(H) \leqslant n$, this implies that the trivial $k[H]$-module $k$ admits a projective resolution
    \[
        0 \rightarrow P_n \rightarrow \cdots \rightarrow P_1 \rightarrow P_0 \rightarrow k \rightarrow 0,
    \]
    where $P_i$ is finitely generated for all $i < n$. By definition, the sequence
    \[
        0 \rightarrow \H_n^{(2)}(H;k) \rightarrow P_n \otimes_{k[H]} \Dk{H} \rightarrow P_{n-1} \otimes_{k[H]} \Dk{H}
    \]
    is exact. But $\H_n^{(2)}(H;k) = 0$ by \cite[Lemma 3.21]{FisherMorales_HNC}, so $P_n \otimes_{k[H]} \Dk{H}$ is finite-dimensional as a right $\Dk{H}$-vector space, being a subspace of the finite-dimensional $\Dk{H}$-vector space $P_{n-1} \otimes_{k[H]} \Dk{H}$. By \cite[Lemma 3.1]{JaikinLinton_coherence}, this implies that $P_n$ is finitely generated, and hence that $H$ is of type $\FP_n(k)$. We conclude that $G$ is homologically $(n-1)$-coherent over $k$.
\end{proof}

\section{Proof of the main result}\label{sec:mains}

\begin{defn}[The Novikov ring] \label{def:nov}
    Let $\chi \colon G \rightarrow \R$ be a nontrivial homomorphism from a group $G$ to the additive group $\R$. The \emph{Novikov ring} of $k[G]$ with respect to $\chi$, denoted $\nov{G}$, is the set of formal series $\sum_{g \in G} \lambda_g g$ with $\lambda_g \in k$ such that
    \[
        \left|\{g \in G : \lambda_g \neq 0 \ \text{and} \ \chi(g) \leqslant r \}\right| < \infty
    \]
    for all $r \in \R$. The obvious addition and multiplication of elements in $\nov{G}$ endows $\nov{G}$ with the structure of a ring.
\end{defn}

The following proposition is a cohomological analogue of \cite[Theorem 5.2]{KielakRFRS} and is proved similarly, though it has been sharpened so that we only require the chain complex to be finitely generated in a single dimension. It relates the vanishing of $\ell^2$-cohomology with the vanishing of Novikov cohomology of a finite-index subgroup. A (not necessarily square) matrix is in \emph{Smith normal form} if it has entries equal to $1$ in some of its diagonal slots, and has entries equal to $0$ in all other slots.

\begin{prop}\label{prop:L2cohomNov}
    Let $G$ be a finitely generated RFRS group and let $P_\bullet$ be a chain complex of projective $k[G]$-modules such that \(P_n\) is finitely generated for some \(n \in \Z\). If 
    \[
        \H^n(\Hom_{k[G]}(P_\bullet, \Dk{G})) = 0,
    \]
    then there is a finite-index subgroup $H \leqslant G$ and an antipodally symmetric open subset $U \subseteq \H^1(H; \mathcal \R)$ such that 
    \begin{enumerate}
        \item $\H^1(G;\R) \subseteq \overline{U}$ and
        \item $\H^n(\Hom_{k[H]}(P_\bullet, \widehat{k[H]}^\chi)) = 0$ for all $\chi \in U$.
    \end{enumerate}
\end{prop}

\begin{proof}
    We may assume that the chain complex consists of free modules \(P_i\), where \(P_n\) is finitely generated, by the following standard argument. Let \(S_n\) be a projective module such that \(P_n \oplus S_n\) is finitely generated and free. Then computing the cohomology of the chain complex
    \[
        \cdots \rightarrow P_{n+2} \rightarrow P_{n+1} \oplus S_n \rightarrow P_n \oplus S_n \rightarrow P_{n-1} \rightarrow \cdots
    \]
    with arbitrary coefficients yields the same result as before complementing with \(S_n\). Let \(S_{n-1}\) be a projective such that \(P_{n-1} \oplus S_{n-1}\) is free. Then we may compute cohomology using the complemented chain complex
    \[
        \cdots \rightarrow P_{n+2} \rightarrow P_{n+1} \oplus S_n \rightarrow P_n \oplus S_n \rightarrow P_{n-1} \oplus S_{n-1} \rightarrow P_{n-2} \oplus S_{n-1} \rightarrow \cdots
    \]
    of projective modules. By iteratively taking complements of the projective modules, we may assume that they are all free.

    For any \(k[G]\)-module \(M\), we denote \(\Hom_{k[G]}(M,\Dk{G})\) by \(M^*\). For every module \(P_i\), we fix a decomposition \(P_i \cong \bigoplus_{I_i} k[G]\) and say that a submodule \(Q_i\) of \(P_i\) is a \emph{standard summand} of \(P_i\) if and only if there is a subset \(J_i \subseteq I_i\) such that \(Q_i \cong \bigoplus_{J_i} k[G]\) under the decomposition.
    
    Consider the portion \(P_{n+1} \rightarrow P_n \rightarrow P_{n-1}\) of the chain complex. Since all the modules are free and \(P_n\) is finitely generated, there is a finitely generated standard summand \(Q_{n-1} \leqslant P_{n-1}\) that contains the image of \(P_n\). It follows that the maps \(Q_{n-1}^* \rightarrow P_n^*\) and \(P_{n-1}^* \rightarrow P_n^*\) have the same image.
    
    We now want to replace \(P_{n+1}\) by a suitably chosen finitely generated standard summand \(Q_{n+1} \leqslant P_{n+1}\) such that the maps \(P_n^* \rightarrow Q_{n+1}^*\) and \(P_n^* \rightarrow P_{n+1}^*\) have the same kernel. Since \(P_{n+1}\) is the directed union of its finitely generated standard summands, 
    \[
        \ker(P_n^* \rightarrow P_{n+1}^*) = \bigcap \ker(P_n^* \rightarrow Q_{n+1}^*),
    \]
    where the intersection is taken over all finitely generated standard summands \(Q_{n+1}\) of \(P_{n+1}\). The system of subspaces \(\ker(P_n^* \rightarrow Q_{n+1}^*)\) is closed under finite intersections, and since they are all finite-dimensional (as \(\Dk{G}\)-modules), it follows that there is a finitely generated standard summand \(Q_{n+1} \leqslant P_{n+1}\) such that 
    \[
        \ker(P_n^* \rightarrow P_{n+1}^*) = \ker(P_n^* \rightarrow Q_{n+1}^*),
    \]
    as desired. It follows that the cohomology of \(Q_{n+1} \rightarrow P_n \rightarrow Q_{n-1}\) in degree \(n\) with coefficients in \(\Dk{G}\) coincides with that of \(P_{n+1} \rightarrow P_n \rightarrow P_{n-1}\) (i.e.~it vanishes).

    Since \(Q_{n\pm1}^*\) and \(P_n^*\) are finitely generated \(\Dk{G}\)-modules, the coboundary maps
    \[
        \delta^n \colon Q_{n-1}^* \rightarrow P_n^* \quad \text{and} \quad \delta^{n+1} \colon P_n^* \rightarrow Q_{n+1}^*
    \]
    can be identified with finite matrices over \(\Dk{G}\). In fact, their entries naturally lie in \(k[G]\), since the maps are induced by boundary maps between free \(k[G]\)-modules. Because \(\Dk{G}\) is a division ring, there are invertible matrices \(M_i\) over \(\Dk{G}\) such that \(M_{i+1} \delta^{i+1} M_i\inv\) is in Smith normal form for \(i = n-1\) and \(i = n\).

    For every normal subgroup \(H \leqslant G\) of finite-index, \(Q_{n+1} \rightarrow P_n \rightarrow Q_{n-1}\) can also be viewed as a chain complex of finitely generated free \(k[H]\)-modules. Moreover, \(\Dk{G}\) is isomorphic to the crossed product \(\Dk{H} * (G/H)\) (see \cite[Proposition 2.2]{JaikinZapirain2020THEUO}). Importantly, this means that $\Dk{G}$ is a $[G:H]$-dimensional vector space over $\Dk{H}$, so the matrices \(M_i\) and their inverses become matrices over \(\Dk{H}\). An \(m\)-by-\(l\) matrix becomes a \([G:H]m\)-by-\([G:H]l\) matrix and each matrix \(M_{i+1} \delta^{i+1} M_i\inv\) is still in Smith normal form after passing to the finite-index subgroup \(H\). Indeed, the entries equal to \(1\) get replaced by \([G:H]\)-by-\([G:H]\) identity matrices, and those equal to zero get replaced by \([G:H]\)-by-\([G:H]\) matrices of zeros.

    By \cite[Theorem 4.13]{KielakRFRS} (and the appendix to \cite{JaikinZapirain2020THEUO} for the positive characteristic case), there is a finite-index normal subgroup \(H \leqslant G\) and an antipodally symmetric open set \(U\) of $\H^1(H;\R)$ such that
    \begin{enumerate}
        \item $\H^1(G;\R) \subseteq \overline U$, and 
        \item for \(i = n-1\) and \(i = n\), the matrices \(M_i^{\pm}\) can be viewed as lying over the Novikov ring \(\nov{H}\) for every \(\chi \in U\).
    \end{enumerate}
    These matrices put the coboundary maps of
    \[\label{eq:novikov_complex}
        \Hom_{k[H]}(Q_{n+1},\nov{H}) \leftarrow \Hom_{k[H]}(P_n,\nov{H}) \leftarrow \Hom_{k[H]}(Q_{n-1},\nov{H}) \tag{\(\dagger\)}
    \]
    into the same Smith normal form as those of \(Q_{n+1}^* \leftarrow P_n^* \leftarrow Q_{n-1}^*\) (viewed as a chain complex of free \(\Dk{H}\)-modules). But then \eqref{eq:novikov_complex} has vanishing homology in degree \(n\). To conclude, observe that the images of \(\Hom_{k[H]}(Q_{n-1},\nov{H})\) and \(\Hom_{k[H]}(P_{n-1},\nov{H})\) in \(\Hom_{k[H]}(P_n,\nov{H})\) are equal and that
    \[
        \ker\left(\Hom_{k[H]}(P_n,\nov{H}) \rightarrow \Hom_{k[H]}(P_{n+1},\nov{H})\right)
    \]
    is contained in
    \[
        \ker\left(\Hom_{k[H]}(P_n,\nov{H}) \rightarrow \Hom_{k[H]}(Q_{n+1},\nov{H})\right).
    \]
    Hence, the degree \(n\) homology of
    \[
        \Hom_{k[H]}(P_{n+1},\nov{H}) \leftarrow \Hom_{k[H]}(P_n,\nov{H}) \leftarrow \Hom_{k[H]}(P_{n-1},\nov{H})
    \]
    is a submodule of the degree \(n\) homology of \eqref{eq:novikov_complex}, and therefore it vanishes. \qedhere
\end{proof}

\begin{cor}\label{cor:l20_nov0}
    Let $G$ be a RFRS group such that the trivial \(k[G]\)-module \(k\) admits a projective resolution \(P_\bullet\) that is finitely generated in degree \(n\) for some \(n \in \N\). Then there is a finite-index subgroup $H \leqslant G$ and an antipodally symmetric open subset $U \subseteq \H^1(H;\R)$ such that
    \begin{enumerate}
        \item $\H^1(G; \R) \subseteq \overline{U}$ and 
        \item $\H^n(H; \nov{H}) = 0$ for each $\chi \in U$.
    \end{enumerate}
\end{cor}

We follow the proof of \cite[Theorem 3.3]{JaikinLinton_coherence} for the following lemma.

\begin{lem}\label{lem:weak_top_dim}
    Let $G$ be RFRS and suppose that $G$ fits into a short exact sequence
    \[
        1 \rightarrow N \rightarrow G \rightarrow \Z \rightarrow 1
    \]
    where $\cd_k(N) = n-1$. Then $\Dk{G}$ is of weak dimension at most $n-1$ as a $k[G]$-module. In particular, $\b{n}(G;k) = 0$.
\end{lem}

\begin{proof}
    Let $M$ be an arbitrary $k[G]$-module. Note that $\Dk{G} \cong \Ore(\Dk{N} * \Z)$, since twisted polynomials rings are Ore domains. Then
    \[
        \Tor_n^{k[G]}(\Dk{N} * \Z, M) \cong \Tor_n^{k[N]}(\Dk{N}, M) = 0
    \]
    by Shapiro's lemma and the fact that $k[N]$ is of global dimension $n-1$. Since Ore localisation is flat, we obtain $\Tor_n^{k[G]}(\Dk{G}, M) = 0$ as desired. \qedhere
\end{proof}

Fix a nontrivial character $\chi \colon G \rightarrow \R$. The next theorem gives a sufficient cohomological criterion for when $\cd_R(\ker \chi) < \cd_R(G)$ for an arbitrary ring \(R\). We thank Andrei Jaikin-Zapirain for communicating a simplification of our original proof of this theorem.

\begin{thm}\label{thm:cd_drop_nov}
    Let \(R\) be a ring and let $G$ be a group such that the trivial \(R[G]\)-module \(R\) admits a  projective resolution of length \(n\) that is finitely generated in degree \(n\). If $\chi \colon G \rightarrow \R$ is a nontrivial character such that
    \[
        \H^n(G; \widehat{R[G]}^{\pm\chi}) = 0,
    \]
    then $\cd_R(\ker \chi) < n$.
\end{thm}

\begin{proof}
    Let $N = \ker \chi$ and let $T$ be a transversal for $N$ in $G$. We first show that $\H^n(N;M) = 0$ for any $R[N]$-module $M$, and therefore that $\cd_R(N) < n$ by \cref{prop:hd_cd}. By Shapiro's lemma, we have
    \[
        \H^n(N;M) \cong \H^n(G; \Hom_{R[N]}(R[G], M)) \cong \H^n(G; \mathcal M)
    \]
    where $\mathcal M = \Hom_{R[N]}(R[G], M) \cong \prod_T M$ is a left $R[G]$-module: the elements of $N$ act factor wise via the action of $N$ on $M$ and the elements of $T$ permute the factors.
    
    Define 
    \[
        \mathcal M^\chi = \{ (x_t)_{t \in T} \in \mathcal M : \text{there exists} \ \alpha \in \R \ \text{such that} \ x_t = 0 \ \text{if} \ \chi(t) < \alpha \}
    \]
    and
    \[
        \mathcal M^{-\chi} = \{ (x_t)_{t \in T} \in \mathcal M : \text{there exists} \ \alpha \in \R \ \text{such that} \ x_t = 0 \ \text{if} \ \chi(t) > \alpha \}.
    \]
    Note that $\mathcal M^{\pm\chi}$ is a left $\widehat{R[G]}^{\pm\chi}$-module.

    We claim that the cohomology of $G$ with coefficients in $\mathcal M^{\pm\chi}$ vanishes. Let \(P_\bullet \rightarrow R\) be a projective resolution of length \(n\) such that \(P_n\) is finitely generated. Then
    \begin{align*}
        \Hom_{R[G]}(P_n,\widehat{R[G]}^\chi) \otimes_{\widehat{R[G]}^\chi} \mathcal M^\chi &\cong \Hom_{R[G]}(P_n,R[G]) \otimes_{R[G]} \widehat{R[G]}^\chi \otimes_{\widehat{R[G]}^\chi} \mathcal M^\chi \\
        &\cong \Hom_{R[G]}(P_n,R[G]) \otimes_{R[G]} \mathcal M^\chi \\
        &\cong \Hom_{R[G]}(P_n,\mathcal M^\chi),
    \end{align*}
    where the first and third isomorphisms follow from the  canonical isomorphism
    \[
        \Hom_S(P,S) \otimes_S L \cong \Hom_S(P,L), \quad f \otimes l \mapsto [x \mapsto f(x) \cdot l],
    \]
    which holds for any ring $S$, finitely generated projective left $S$-module $P$, and arbitrary left $S$-module $L$ (this fails if $P$ is not finitely generated).
    
    The commutative diagram
    \[
        \begin{tikzcd}
            {\Hom_{R[G]}(P_n,\widehat{R[G]}^\chi) \otimes_{\widehat{R[G]}^\chi} \mathcal M^\chi} \arrow[d, "\cong"'] \arrow[r] & {\H^n(G;\widehat{R[G]}^\chi) \otimes_{\widehat{R[G]}^\chi} \mathcal M^\chi} \arrow[d] \arrow[r] & 0 \\
            {\Hom_{R[G]}(P_n,\mathcal M^\chi)} \arrow[r]  & \H^n(G;\mathcal M^\chi) \arrow[r] & 0
        \end{tikzcd}
    \]
    has exact rows, and therefore the rightmost vertical map is surjective. By assumption, \(\H^n(G;\widehat{R[G]}^{\chi}) = 0\), so \(\H^n(G;\mathcal M^\chi) = 0\) as well. By a similar argument, \(\H^n(G;\mathcal M^{-\chi}) = 0\), as claimed.

    The $R[G]$-module inclusions of $\mathcal M^{\pm\chi} \hookrightarrow \mathcal M$ induce a surjection 
    \[
        \mathcal M^\chi \oplus \mathcal M^{-\chi} \rightarrow \mathcal M.
    \]
    Denote the kernel of this map by $\mathcal K$. Then we have a short exact sequence
    \[
        0 \rightarrow \mathcal K \rightarrow \mathcal M^\chi \oplus \mathcal M^{-\chi} \rightarrow \mathcal M \rightarrow 0
    \]
    of $R[G]$-modules. The long exact sequence in $\H^\bullet(G;-) = \Ext^\bullet_{R[G]}(R,-)$ contains the portion
    \[
        \H^n(G; \mathcal M^\chi \oplus \mathcal M^{-\chi}) \rightarrow \H^n(G; \mathcal M) \rightarrow \H^{n+1}(G; \mathcal K).
    \]
    But $\H^{n+1}(G; \mathcal K) = 0$ because $\cd_R(G) \leqslant n$ and 
    \[
        \H^n(G; \mathcal M^\chi \oplus \mathcal M^{-\chi}) \cong \H^n(G; \mathcal M^\chi) \oplus \H^n(G; \mathcal M^{-\chi}) = 0
    \]
    by the above. Hence, $\H^n(G; \mathcal M) = \H^n(N; M) = 0$. \qedhere
\end{proof}

We now arrive at our main result.

\begin{thm}\label{thm:weak_fibre}
    Let $G$ be a nontrivial finitely generated RFRS group such that there is a projective resolution
    \[
        0 \rightarrow P_n \rightarrow \cdots \rightarrow P_0 \rightarrow k \rightarrow 0
    \]
    of the trivial \(k[G]\)-module \(k\) and \(P_n\) is finitely generated. The following are equivalent:
    \begin{enumerate}
        \item\label{item:L2vanish} $\b{n}(G;k) = 0$;
        \item\label{item:weakFibre} there is a finite-index subgroup $H \leqslant G$ and an epimorphism $\chi \colon H \rightarrow \Z$ such that $\cd_k(\ker \chi) \leqslant n-1$;
        \item\label{item:weakFibre_hd_version} there is a finite-index subgroup $H \leqslant G$ and an epimorphism $\chi \colon H \rightarrow \Z$ such that $\hd_k(\ker \chi) \leqslant n-1$.
    \end{enumerate}
    If any of the above conditions hold and additionally there is an integer \(m\) such that \(G\) is of type \(\FP_m(k)\) and \(\b{i}(G;k) = 0\) for all \(i \leqslant m\), then there is a finite-index subgroup \(H \leqslant G\) and an epimorphism \(\chi \colon H \rightarrow \Z\) such that \(\ker \chi\) is of type \(\FP_m(k)\) and \(\cd_k(\ker \chi) \leqslant n-1\).
\end{thm}

If we additionally assume $\cd_k(G) = n$, then we would conclude that $\hd_k(G) = \cd_k(\ker \chi) = n-1$ by subadditivity of (co)homological dimension under group extensions. This gives the statement of \cref{thm:weakFibre}.

\begin{proof}[Proof (of \cref{thm:weak_fibre})]
    Suppose that \ref{item:L2vanish} holds. By \cref{cor:l20_nov0}, there is a finite-index subgroup $H \leqslant G$ and an open set $U \subseteq \H^1(H;\R)$ such that $\H^1(G; \R) \subseteq \overline U$ and
    \[
        \H^n(H; \widehat{k[H]}^{\pm\chi}) = 0
    \]
    for all characters $\chi \in U$. Fix a nontrivial character $\chi \in U$ such that $\chi(H) \subseteq \Z$ (there are many such characters since $U$ is open). Then \(\cd_k(\ker \chi) \leqslant n-1\) by \cref{thm:cd_drop_nov}.
    
    If \ref{item:weakFibre} holds, then it is clear that \ref{item:weakFibre_hd_version} holds because the inequality \(\hd_k(\Gamma) \leqslant \cd_k(\Gamma)\) holds for any group \(\Gamma\). If \ref{item:weakFibre_hd_version} holds, then $\b{n}(\ker \chi;k) = 0$ trivially, and hence $\b{n}(H;k) = 0$ by \cite[Theorem 6.4]{Fisher_Improved}. Thus \ref{item:L2vanish} holds by \cite[Lemma 6.3]{Fisher_Improved}.
    
    If, additionally, \(\b{i}(G;k) = 0\) for all \(i \leqslant m\) and \(G\) is of type \(\FP_m(k)\), then $G$ virtually fibres with kernel of type $\FP_m(k)$ by \cite[Theorem B]{Fisher_Improved}. Since the conclusion we want to prove is virtual, we will assume there is an epimorphism $G \rightarrow \Z$ with kernel of type $\FP_m(k)$. As mentioned above, there is a finite-index subgroup \(H \leqslant G\) and an antipodally symmetric open set \(U \subseteq \H^1(H;\R)\) whose closure contains every character on \(G\) and such that the top-dimensional Novikov cohomology vanishes for all characters in \(U\). Since fibring with kernel of type \(\FP_m(k)\) is an open condition (see \cite[Theorem A]{BieriRenzValutations}), there is then an epimorphism \(\chi \colon H \rightarrow \Z\) such that \(\ker \chi\) is of type \(\FP_m(k)\) and \(\H^n(H; \widehat{k[H]}^{\pm\chi})=0\). We conclude that \(\cd_k(\ker \chi) \leqslant n-1\) by \cref{thm:cd_drop_nov}. \qedhere
\end{proof}

\begin{rem}
    \begin{enumerate}
        \item Groups of type \(\FP(k)\) and of cohomological dimension at most \(n\) over \(k\) are examples of groups where \(k\) admits a projective resolution as in \cref{thm:weak_fibre}. Moreover, if \(G\) is RFRS of type \(\FP_{n-1}(k)\) with \(\cd_k(G) = n\) and \(\b{n}(G;k) = 0\), then \(G\) is of type \(\FP(k)\) by \cref{prop:hom_n_coherence}. This is why \cref{thm:main} holds for finitely generated RFRS groups without the type \(\FP_2(k)\) assumption.
        \item The assumption that the RFRS group \(G\) be finitely generated is only used in \cref{prop:L2cohomNov} when we apply \cite[Theorem 4.13]{KielakRFRS} to show that matrices over \(\Dk{G}\) can be represented over \(\widehat{k[H]}^\chi\) for some finite-index subgroup \(H \leqslant G\). It is actually enough to assume that \(G\) is countable and residually (poly-\(\Z\) and virtually Abelian) for this to hold (note that a finitely generated group is RFRS if and only if it is residually (poly-\(\Z\) and virtually Abelian) by \cite[Theorem 6.3]{OkunSchreve_DawidSimplified}). This is implicit in the work of Okun--Schreve \cite{OkunSchreve_DawidSimplified} and is made explicit in a forthcoming article of Kevin Klinge and the author, where the class of residually (poly-\(\Z\) and virtually nilpotent) groups is studied. For example, if \(G\) is a RAAG on a countably infinite graph \(\Gamma\) such that the flag completion of \(\Gamma\) is finite-dimensional and has finitely many top-dimensional simplices, then \(G\) satisfies all the hypotheses of \cref{thm:weak_fibre} except that it is not finitely generated. Moreover, it is easy to construct examples where the top-degree \(\ell^2\)-Betti number vanishes, and therefore such examples would admit virtual maps to \(\Z\) with kernels of strictly lower cohomological dimension.
    \end{enumerate}
\end{rem}

For RFRS groups \(G\) with \(\b{n}(G;k) = 0\) (where \(n = \cd_k(G)\)), we can now strengthen the homological \((n-1)\)-coherence of \(G\) over \(k\) to \((n-1)\)-coherence of the group algebra \(k[G]\). Note that $(n-1)$-coherence of $k[G]$ implies $(n-1)$-homological coherence of $G$ over $k$. Whether the converse holds is open, even for $n = 2$.

\begin{cor}\label{cor:group_alg_coherence}
    Let $G$ be a RFRS group of type $\FP_{n-1}(k)$ with $\cd_k(G) = n$. If $\b{n}(G;k) = 0$, then the group algebra $k[G]$ is $(n-1)$-coherent.
\end{cor}
\begin{proof}
    By \cref{prop:hom_n_coherence}, \(G\) is of type \(\FP(k)\), and therefore \cref{thm:weak_fibre} implies there is a virtual map to \(\Z\) of cohomological dimension \(n-1\) over \(k\). Since group algebra coherence is a commensurability invariant, we may assume there exists a map $\chi \colon G \rightarrow \Z$ such that $\cd_k(\ker \chi) = n-1$. Then $\Dk{G}$ is of weak dimension at most $n-1$ by \cref{lem:weak_top_dim}, which implies that $k[G]$ is $(n-1)$-coherent by \cref{prop:ring_coherence}. \qedhere
\end{proof}

We now specialise to dimension $2$, where \cref{thm:weak_fibre} has the strongest consequences.

\begin{cor}\label{cor:2dim}
    Let $G$ be a finitely generated RFRS group with $\cd_k(G) \leqslant 2$. Then $G$ is virtually free-by-cyclic if and only if $\b{2}(G;k) = 0$.
\end{cor}
\begin{proof}
    This follows immediately from \cref{thm:weak_fibre} applied at $n = 2$ and the Stallings--Swan theorem \cite{Stallings_cd1,Swan_cd1}, which states that torsion-free groups of cohomological dimension one over \(k\) are free. \qedhere
\end{proof}

\begin{cor}
    If $G$ is RFRS, $\cd_k(G) \leqslant 2$, and $\b{2}(G;k) = 0$, then $G$ and $k[G]$ are coherent. If $G$ is finitely generated, then $G$ has a finite-index subgroup $H \cong \pi_1(X)$ where $X$ is a $2$-complex with nonpositive immersions.
\end{cor}
\begin{proof}
    Let $H \leqslant G$ be a finitely generated subgroup. Then $\b{2}(H;k) = 0$ by \cite[Lemma 3.21]{FisherMorales_HNC}, so $H$ is virtually free-by-cyclic by \cref{cor:2dim} and hence finitely presented by \cite{FeighnHandel_FreeByZCoherent}.

    Let $I \trianglelefteqslant k[G]$ be a finitely generated (left) ideal with finite generating set $S \subseteq k[G]$. Note that there is a finitely generated subgroup $H \leqslant G$ such that $S \subseteq k[H]$. Then $H$ is virtually free-by-cyclic by the argument above. Let $I_H$ be the left ideal of $k[H]$ generated by $S$. It follows from \cite[Theorem B, Proposition 2.9]{HennekeLopez_PseudoSyl} that group algebras of free-by-cyclic groups are coherent, so there is an exact sequence
    \[
        k[H]^m \rightarrow k[H]^n \rightarrow I_H \rightarrow 0
    \]
    for some integers $m$ and $n$. Since $k[G]$ is free as a $k[H]$-module, it is also flat and thus tensoring with $k[G]$ gives the exact sequence
    \[
        k[G]^m \rightarrow k[G]^n \rightarrow k[G] \otimes_{k[H]} I_H \rightarrow 0.
    \]
    We verify that the multiplication map $m \colon k[G] \otimes_{k[H]} I_H \rightarrow I$ is an isomorphism, which will conclude the proof that $k[G]$ is coherent. The generating set $S$ is clearly in the image of $m$, so $m$ is surjective. Let $T$ be a left transversal for $H$ in $G$. Then 
    \[
        k[G] \otimes_{k[H]} I_H \cong \bigoplus_{t \in T} t \cdot k[H] \otimes_{k[H]} I_H \cong \bigoplus_{t \in T} t \cdot I_H \subseteq \bigoplus_{t \in T} t \cdot k[H] \cong k[G],
    \]
    and therefore $m$ is injective. 
    
    The final claim follows from \cref{cor:2dim} and \cite[Theorem 6.1]{Wise_JussieuCoherenceNPI}, which states that ascending HNN extensions of free groups are fundamental groups of $2$-complexes with nonpositive immersions.
    \qedhere
\end{proof}

\bibliography{bib}
\bibliographystyle{alpha}

\end{document}